\documentclass{article}
\usepackage{graphicx} 
\usepackage{arxiv}
\usepackage{amsmath,amssymb}
\usepackage{amsthm}
\newtheorem{lem}{Lemma}
\newtheorem{thm}{Theorem}
\newtheorem{cor}{Corallary}

\title{List Coloring of some Cayley graphs using Kernel perfections}
\author{Prajnanaswaroopa S\\
sntrm4@rediffmail.com}
\date{September 2023}

\begin{document}

\maketitle

\section*{Abstract}
In this paper, we try to determine exact or bounds on the choosability, or list chromatic numbers of some Cayley graphs, typically some Unitary Cayley graphs and Cayley graphs on Dihedral groups.
\section*{Introduction}
Kernels in digraphs were used to study games on graphs. Later, they were used to study list colorings. Specifically,  \cite{KER1}, \cite{KER2} gives some proper introdution and applications of the kernels and kernel-perfect orientations. In this paper, we discuss the list coloring of some Cayley graphs using kernel perfect orientations, which were first applied in the list coloring setting by Galvin to prove the List Coloring conjecture restricted to bipartite graphs \cite{GAL}. 

A kernel is an independent set $S$ in a digraph $G$ such that all vertices in $G/S$ have an edge directed towards $S$. Thus, it can be said to be a directed independent dominating set of the digraph $G$. A digraph $G$ is said to be kernel-perfect if every induced subgraph of $G$ has a kernel. A kernel-perfect orientation of an undirected graph $G$ is orienting the edges of $G$ such that it becomes a kernel-perfect digraph.
\section*{Circulant graphs}
A Cayley graph is an interesting graph associated with an algebraic structure. Consider a group/groupoid/magma $\Gamma$, with a symmetric generating set $S$ (A set such that $s\in S$ iff $s^{-1}\in S$) without the identity of $\Gamma$. The graph having vertices as all the elements of $\Gamma$, with edges determined by the symmetric generating set $S$, such that two elements of the group $x,y$ are adjacent iff $x=ys$ is called a Cayley graph.

A Cayley graph on $\Gamma$ when $\Gamma$ is a cyclic group is called a Circulant graph. Among circulant graphs, if the generating set consists of all the elements coprime to the order of the graph, the Cayley graph is said to be a Unitary Cayley graph. A circulant graph is a $k$-th power of an $n$ cycle if the generating set $S$ is of the form $\{1,2,\ldots,k,n-k,\ldots,n-2,n-1\}$.

It is well known that the Powers of Cycles are chromatic choosable \cite{WOOD}. Here we try to obtain similar results for unitary Cayley graphs. The object will be to find a kernel-perfect orientation of some unitary Cayley graphs. The following lemma follows from the work of \cite{BOR}, \cite{BOR1}. 

\begin{lem}
If we can orient a perfect graph $G$ in a kernel solvable manner with maximum outdegree $k$, then $ch(G)\le k+1$. 
\end{lem}
\begin{proof}
Let $D$ be the kernel solvable orientation of the graph $G$. Since $G$ has no odd cycles of length $\ge5$, therefore, in its kernel solvable orientation $D$, we will not have a directed cycle of length $\ge5$, which implies that once we have a kernel solvable orientation, we also have a kernel perfect orientation with maximum outdegree $l$, thus implying $ch(G)\le l$. 
\end{proof}

\begin{thm}
The Unitary Cayley Graphs $G$ of order $n$, where $n$ is odd has at most two prime divisors satisfies $ch(G)=k$, where $k$ is the least prime dividing $n$.
\end{thm}
\begin{proof}
First, from \cite{KLOT}, the graphs $G$ are perfect. We create a digraph $D$ from $G$ as follows. We orient all the Hamiltonian cycles generated by the elements $1,2,\ldots,k$, where $k$ is the least prime dividing $n$ in the anticlockwise sense. We can observe that in this orientation, all the cliques (which are precisely the sequence of any $k$ consecutive vertices) are given a kernel, that is, an all-absorbing vertex. Thus, $D$ is kernel solvable with maximum outdegree $k-1$.
Hence, it directly follows from the previous lemma that $ch(G)\le (k-1)+1=k$. As $G$ contains $k$ clique, therefore, $ch(G)=k.$
\end{proof}
\begin{thm}
The complement of Unitary Cayley graphs $G$ of odd order $n$, where $n$ has at most two prime divisors satisfies $ch(G)=\chi(G)$
\end{thm}
\begin{proof}
Here, by \cite{KLOT}, and the fact that a graph is perfect if and only if its complement is perfect, we get that $G$ is perfect. In addition, we see that the clique number equals $\frac{n}{p}$, where $p$ is the smallest prime divisor of $n$. The clique is formed by the elements of the form $a-(a+p)-(a+2p)-\ldots-(a+n-p)$, where addition is modulo $n$. This also implies, by perfectness, that $chi(G)=\frac{n}{p}$. Now, we can give a kernel solvable orientation with maximum outdegree $\frac{n}{p}-1$ as follows. We orient all the Hamiltonian cycles generated by  $p, 2p, \ldots, n-p$ an orientation in the anticlockwise sense. Then, we can see that the orientation given is kernel-perfect (actually kernel solvable). The maximum outdegree, which is constant for every vertex is $\frac{n}{p}-1$. therefore, we can conclude that $ch(G)=\frac{n}{p}$, which proves the theorem.
\end{proof}

\section*{Cayley graphs on Dihedral groups}
The dihedral groups, denoted in this paper by $D_{2n}$, are the groups that consist of symmetries of a regular  $n$-gon. The minimal generators of dihedral groups in the following are $\{r,s\}$, where $r$ is a reflection element satisfying $r^2=1$ and $s$ is a rotation element having order $n$ with the usual group defining the relation $(rs)^2=e$. The rotation element generates the cyclic group of order $n$, with elements represented by $\{0,1,2,\ldots,n-1\}$.

As $1$ is always a generating element for the cyclic group $\mathbb{Z}_n$, let us assume $s=1$. Then, we could write the elements $rs^a$ as $ra$. We use this notation in the following theorems.

 We try to determine a result similar to the one for circulant graphs for some Cayley graphs on Dihedral groups. 
\begin{thm}
If $n$ is odd and has at most two odd prime divisors, then the graph $G=C(D_{2n}, S)$ with $S=S_1\cup\{r,r1\}$, where $S_1$ is the set of numbers coprime to $n$, satisfies $ch(G)\le p+1$, where $p$ is the least prime dividing $n$.  
\end{thm}
\begin{proof}
We see that the induced subgraph formed by the generating set $S_1$ is the disjoint union of two unitary Cayley graphs on $n$ vertices. By the previous theorem, we know that these graphs are chromatic choosable, with their choice numbers equal to $p$. In addition, these subgraphs can be given a kernel perfect (kernel solvable) orientation with  maximum outdegree equal to $p-1$. Now, when we form the edges generated by the elements $\{r,r1\}$, these edges induced a bipartite graph with partite sets being the vertex sets of the disjoint unitary Cayley graphs of order $n$ induced by the set $S_1$ earlier. Thus, these edges can always be given a kernel perfect orientation, which, together with the kernel perfect orientation of the unitary Cayley graphs induced by $S_1$, would give us a kernel perfect orientation of $G$ with maximum outdegree $p$. Hence, $ch(G)\le p+1$.

\end{proof}
 \begin{cor}
 If $n$ is odd and has at most two odd prime divisors, then the graph $G=C(D_{2n}, S)$ with $S=S_1\cup S_2$, where $S_1$ is the set of numbers coprime to $n$, and $S_2$ is a set of even cardinality from the set $\{1,2,\ldots,n\}-S_1$, satisfies $ch(G)\le p+k$, where $p$ is the least prime dividing $n$, and $k=\frac{\mid S_2\mid}{2}$.
 \end{cor}
 \begin{proof}
 The proof is immediate once we use the previous theorem and the fact that the edges induced by the set $S_2$ can be given a kernel perfect orientation with maximum outdegree $k$. This is achieved by orienting the edges alternately outwards and inwards for every vertex.
 \end{proof}
\begin{thm}
If $n$ is odd and has at most two odd prime divisors, then the graph $G=C(D_{2n}, S)$ with $S=S_1\cup\{r,r1\}$, where $S_1$ is the complement of the set of numbers coprime to $n$, satisfies $ch(G)\le \frac{n}{p}+1$, where $p$ is the least prime dividing $n$.  
\end{thm}
\begin{proof}
The proof is very similar to the previous theorem. In this case, the graph induced by the generating set $S_1$ is a disjoint union of two graphs, which are each complement of unitary Cayley graphs on $n$ vertices. Hence, the two disjoint copies can be given a kernel perfect orientation with maximum outdegree $\frac{n}{p}-1$. 

The edges induced by the set $\{r,r1\}$ can be always given a kernel perfect orientation in $G$ such that the maximum outdegree of the orientation of $G$ is $\frac{n}{p}$. Hence, $ch(G)\le\frac{n}{p}+1$.
\end{proof}
\begin{cor}
If $n$ is odd and has at most two odd prime divisors, then the graph $G=C(D_{2n}, S)$ with $S=S_1\cup S_2$, where $S_1$ is the complement of the set of numbers coprime to $n$, and $S_2$ is a subset of even cardinality from the set $\{1,2,\ldots,n\}-S_1$,  satisfies $ch(G)\le \frac{n}{p}+k$, where $p$ is the least prime dividing $n$, and $k=\frac{|S_2|}{2}$.  
\end{cor}
\begin{proof}
The proof is exactly similar to that of the corollary of the previous theorem.
\end{proof}
 
 
\end{document}